\documentclass[11pt]{article}
\usepackage{amsmath,amsthm,amsfonts,amscd,amssymb,eucal,latexsym,mathrsfs, appendix, yhmath, fancyhdr, setspace,url}
\usepackage[numbers,sort&compress]{natbib}
\usepackage[all,cmtip]{xy}
\usepackage{abstract}
\newtheorem{theorem}{Theorem}[section]

\newtheorem{lemma}[theorem]{Lemma}
\newtheorem{proposition}[theorem]{Proposition}

\theoremstyle{definition}
\newtheorem{definition}[theorem]{Definition}
\newtheorem{remark}[theorem]{Remark}

\allowdisplaybreaks

\begin{document}

\title{Fibred cofinitely-coarse embeddability of box families\\ and proper isometric affine actions on\\ uniformly convex Banach spaces\footnote{The authors are supported by NSFC (Nos. 11231002, 11771061).}}

\author{Guoqiang Li and  Xianjin Wang}

\date{}

\maketitle

\begin{abstract}
In this paper we show that a countable, residually amenable group admits a proper isometric affine action on some uniformly convex Banach space if and only if one (or equivalently, all) of its box families admits a fibred cofinitely-coarse embedding into some uniformly convex Banach space.

\end{abstract}

\maketitle \numberwithin{equation}{section}
\section{Introduction}
The Haagerup property, also viewed as Gromov's definition of {\it a-T-menability} (definition in terms of isometric affine actions, see \cite{Gromov}), was first introduced in \cite{Haagerup} by Haagerup. Considered as a weak form of von Neumann's amenability and a strong negation of Kazhdan's property $(T)$, the Haagerup property has significant applications in many fields of mathematics, including representation theory, $K$-theory, the Baum-Connes conjecture, harmonic analysis, etc. For example, Higson and Kasparov in their paper \cite{Higson} established the strong Baum-Connes conjecture for groups with the Haagerup property. Specifically, a (discrete) group $G$ is called a-T-menable or has the Haagerup property if there exists a proper action by affine isometries of $G$ on a Hilbert space. The notion of Haagerup property naturally extends to proper isometric affine action on Banach spaces. In recent years, great progress has been made about isometric affine actions on Banach spaces, and particularly in the case of $L^p$ spaces for fixed $1\leq p\leq\infty$. For more on this, see \cite{Arnt}, \cite{Bader}, \cite{Chatterji}, \cite{Cherix1}, \cite{Yu}.

Recall that a finitely generated group $G$ is residually finite if there exists a nested sequence of finite index normal subgroups such that the intersection of these subgroups is trivial. The box space associated to this sequence reflects in some way the structure of the residually finite group. For example, X. Chen, Q. Wang and X. Wang in their paper \cite{Wang} gave the characterization of the Haagerup property in terms of {\it fibred coarse embedding into Hilbert space}: they showed that a finitely generated, residually finite group has the Haagerup property if and only if one of its {\it box spaces} admits a fibred coarse embedding into a Hilbert space. The idea of a fibred coarse embedding into Hilbert space for metric spaces was introduced in \cite{Chen} by X. Chen, Q. Wang, and G. Yu to attack the maximal Baum-Connes conjecture. S. Arnt in \cite{Arnt} extended this result to the class of $L^p$ spaces (for a fixed $p\geq 1$). That is, a finitely generated, residually finite group has {\it the property $PL^p$} (i.e. it admits a proper isometric affine action on some $L^p$ space) if and only if one of its box spaces admits a fibred coarse embedding into some $L^p$ space. K. Orzechowski in \cite{Orzechowski} extended this result for {\it residually amenable groups} (see Definition \ref{def1} below). That is, a countable, residually amenable group has the Haagerup property if and only if one of its {\it box families} (see Definition \ref{def2} below) admits a {\it fibred cofinitely-coarse embedding} (see Definition \ref{def4} below) {\it into a Hilbert space}. Moreover, in \cite{Kasparov} G. Kasparov and G. Yu showed that groups admitting coarse embeddings into uniformly convex Banach spaces satisfy the Novikov conjecture. Inspired by the above works, we extend further the result in \cite{Wang}: for a residually amenable group $G$, we give a sufficient condition of the proper isometric affine action of $G$ on some uniformly convex Banach space by fibred cofinitely-coarse embeddability of its box families.

\begin{theorem}
A countable, residually amenable group admits a proper isometric affine action on some uniformly convex Banach space if and only if one (or equivalently, all) of its box families admits a fibred cofinitely-coarse embedding into a uniformly convex Banach space.~\label{th1}
\end{theorem}

\section{Preliminaries}
First, we recall the definition of the residually amenable group and the box family.
\begin{definition}[see \cite{Orzechowski}]
Let $G$ be a countable, finitely generated group. Then $G$ is residually amenable if there exists a sequence $(G_n)_{n\in\mathbb{N}}$ of normal subgroups in $G$, such that $G_{n+1}\subset G_n$ for every $n\in\mathbb{N}$, $\bigcap_{n\in\mathbb{N}}G_n =\{e_G\}$ and each quotient group $G/G_n$ is amenable.~\label{def1}
\end{definition}
Obviously, all residually finite groups are residually amenable.

Let $G$ be a residually amenable group and $(G_n)_{n\in\mathbb{N}}$ be a sequence of normal subgroups in $G$ satisfying the conditions of Definition \ref{def1}. Endow $G$ with a {\it proper length function} $l$ and its associated, left-invariant metric $d(g,h):=l(g^{-1}h)$ for $g,h\in G$. Then, for each $n\in\mathbb{N}$, $$l_n ([g]):=\inf\{l(gh):h\in G_n\}$$ defines a length function on the quotient group $G/G_n$. Thus, $G/G_n$ becomes a metric space with the metric $d_n$ induced by $l_n$. This way, we can obtain a family of metric spaces.
\begin{definition}[see \cite{Orzechowski}]
Let $G$ be a residually amenable group. The metric family $\{(G/G_n ,d_n):n\in\mathbb{N}\}$ is said to be the box family of $G$ corresponding to $(G_n)_{n\in\mathbb{N}}$.~\label{def2}
\end{definition}

Let $(X_n)_{n\in\mathbb{N}}$ be a sequence of metric spaces. {\it A coarse disjoint union} of $(X_n)_{n\in\mathbb{N}}$ is the disjoint union $X=\bigsqcup_{n\in\mathbb{N}}X_n$ equipped with a metric $d$ defined as follows: (1) the restriction of $d$ to each $X_n$ is the original metric of $X_n$; (2) $d(X_n,X_m)\to\infty$ as $n+m\to\infty$ and $n\neq m$. We remark that the box family $\{(G/G_n):n\in\mathbb{N}\}$ can be considered as the coarse disjoint union $\bigsqcup_{n\in\mathbb{N}}(G/G_n)$. In particular, if $G$ is a finitely generated, residually finite group, then the coarse disjoint union $\bigsqcup_{n\in\mathbb{N}}(G/G_n)$ is called the {\it box space}.

\vspace{3mm}Next we recall the definition of {\it the uniform convexity} and {\it the proper isometric affine action}.

Let $(B,\|\cdot\|_B)$ be a Banach space. {\it The modulus of convexity} of $B$ is the function $\delta_B :[0,2]\to [0,1]$ defined by
$$\delta_B (t):=\inf\bigg\{1-\left\|\frac{\xi+\xi'}{2}\right\|_{B}:\xi,\xi'\in B,\,\|\xi\|_B =\|\xi'\|_B =1,\,\|\xi-\xi'\|_B \geq t\bigg\}.$$
The Banach space $B$ is said to be uniformly convex if $\delta_B (t)>0$ for any $t>0$. Let $(B_i)_{i\in\mathbb{N}}$ be a sequence of uniformly convex Banach spaces. The sequence $(B_i)_{i\in\mathbb{N}}$ is said to have a {\it common modulus of convexity} if $\inf\limits_{i}\delta_{B_i}(t)>0$ for any $t>0$.

M. Day in \cite{Day} proved the following fact.
\begin{theorem}[see \cite{Day}]
Let $(B_i)_{i\in\mathbb{N}}$ be a sequence of Banach spaces, $1<p<\infty$ and define $$\bigoplus_{i=1}^{\infty}B_i:=\bigg\{b=(b_i):b_i \in B_i,\,\|b\|_p :=\bigg(\sum\limits_{i=1}^{\infty}\|b_i\|_{B_i}^p\bigg)^{\frac{1}{p}}<\infty\bigg\}.$$ Then $\bigoplus_{i=1}^{\infty}B_i$ is uniformly convex if and only if the $(B_i)_{i\in\mathbb{N}}$ have a common modulus of convexity.~\label{th2}
\end{theorem}

Let $G$ a countable, amenable group. By one of the equivalent definitions of amenability, there exists {\it a finitely additive probability measure} $\mu:2^G\to [0,1]$ which is right-invariant (i.e.$\,\,\mu(Az)=\mu(A)$ for $A\subset G$, $z\in G$). Let $(B_z)_{z\in G}$ be a sequence of Banach spaces, $1<p<\infty$ and define $$\prod _{z\in G}B_z:=\bigg\{b=(b_z):b_z \in B_z,\,\|b\|_p:=\left(\int_G \|b_i\|_{B_z}^p \,d\mu(z)\right)^{\frac{1}{p}}<\infty\bigg\}.$$ Then Theorem \ref{th2} still holds.

\vspace{3mm}An {\it affine isometry} is a function $A:B\to B$ sending an element $\xi\in B$ to $I(\xi)+b$ where $I$ is a bijective linear isometry and $b\in B$ is a fixed element. The set of affine isometries of $B$ is a group under composition. We denote this group by $\rm{AffIsom}(B)$.
\begin{remark}
By the Mazur-Ulam theorem \cite{Mazur}, any bijective isometry between normed spaces over $\mathbb{R}$ is in fact an affine isometry.~\label{re1}
\end{remark}
An affine isometric action of a group $G$ on a Banach space $B$ is a homomorphism $\alpha$ of $G$ into the group $\rm{AffIsom}(B)$, such an action can be characterized by the following decomposition:
$$\alpha(g)(\xi)=L(g)(\xi)+b(g)$$ for all $g\in G$, $\xi\in B$, where $L$ is an isometric representation of $G$ on $B$ and $b:G\to B$ is a vector-valued map. Since $\alpha$ is a homomorphism, we have $\alpha(gh)=\alpha(g)\alpha(h)$ for every $g,h\in G$. Hence, for any $\xi\in B$,$$L(gh)(\xi)+b(gh)=\alpha(gh)(\xi)=\alpha(g)\alpha(h)(\xi)=L(g)L(h)(\xi)+L(g)b(h)+b(g).$$
Thus,
\begin{equation}\label{equ3}
b(gh)=L(g)b(h)+b(g).
\end{equation}
The map $b$ satisfying the condition \eqref{equ3} is called a {\it cocycle} for $L$. The action $\alpha$ is said to be proper if it satisfies:
\begin{equation}\label{equ2}
\forall\xi\in B,\,\,\forall A\subset B \,\,\rm{bounded}, \,\#\{g\in G:\alpha(g)(\xi)\in A\}<\infty.
\end{equation}
If we equip $G$ with a length function $l$, then the condition \eqref{equ2} is equivalent to the following:
\begin{equation}\label{equ21}
\lim\limits_{l(g)\to\infty}\|b(g)\|_B =\infty.
\end{equation}

The notion of fibred coarse embedding into Hilbert space for metric spaces has been introduced in \cite{Chen}. The Hilbert space in the notion can be replaced by any other metric space $Y$ as a model space.
\begin{definition}[see \cite{Wang}]
A metric space $(X,d_X)$ is said to admit a {\it fibred coarse embedding into a metric space} $(Y,d_Y)$ if there exist:

\vspace{3mm}$\bullet$ a field of metric spaces $(Y_x)_{x\in X}$ such that each $Y_x$ is isometric to $Y$;

\vspace{3mm}$\bullet$ a section $s:X\to\bigsqcup_{x\in X}{Y_x}$ (that is, $s(x)\in Y_x$);

\vspace{3mm}$\bullet$ two non-decreasing functions $\rho_1$ and $\rho_2$ from $[0,\infty)$ to $(-\infty,\infty)$ with
$\lim\limits_{r\to\infty}\rho_{i}(r)=\infty \,(i=1,2)$
\begin{flushleft}
such that, for any $r>0$, there exists a bounded subset $K_{r}\subset
X$ for which there exists a ¡®trivialization¡¯
\end{flushleft}
$$t_{C}:(Y_x)_{x\in C}\to C\times Y$$ for each subset $C\subset
X\backslash K_r$ of diameter less than $r$, that is, a map from $(Y_x)_{x\in C}$ to the constant field $C\times Y$
over $C$ such that the restriction of $t_C$ to the fibre $Y_{x}\,(x\in C)$ is an isometry
$t_C(x):Y_{x}\to Y$, satisfying the following conditions:

\vspace{3mm}(i) for any $x,x'\in C,\,\rho_{1}(d_X (x,x'))\leq d_Y(t_{C}(x)(s(x)),t_{C}(x')(s(x')))\leq\rho_{2}(d_X(x,x'))$;

\vspace{3mm}(ii) for any two subsets $C_1,C_2\subset X\backslash K_r$ of diameter less than r with $C_{1}\cap C_{2}\neq \varnothing$, there exists an isometry $t_{C_{1}C_{2}}:Y\to Y$ such that $t_{C_1}(x)\circ t_{C_2}^{-1}(x)=t_{C_{1}C_{2}}$ for all $x\in C_{1}\cap C_{2}$.~\label{def3}
\end{definition}
An important application of this property is the following result in \cite{Chen}:

\begin{theorem}[see \cite{Chen}]
Let $X$ be a bounded geometry metric space admitting a fibred coarse embedding into Hilbert space. Then the maximal coarse Baum-Connes conjecture holds for $X$.~\label{th3}
\end{theorem}

In \cite{Orzechowski}, K. Orzechowski modified the definition above slightly so that it would be useful for the box family.
\begin{definition}[see \cite{Orzechowski}]
A family $\mathscr{X}$ of disjoint metric spaces is said to admit a {\it fibred cofinitely-coarse embedding into a metric space} $(Y,d_Y)$ if there exist:

\vspace{3mm}$\bullet$ a field of metric spaces $(Y_x)_{x\in \bigcup\mathscr{X}}$ such that each $Y_x$ is isometric to $Y$;

\vspace{3mm}$\bullet$ a section $s:\bigcup\mathscr{X}\to\bigsqcup_{x\in \bigcup\mathscr{X}}{Y_x}$ (that is, $s(x)\in Y_x$);

\vspace{3mm}$\bullet$ two non-decreasing functions $\rho_1$ and $\rho_2$ from $[0,\infty)$ to $(-\infty,\infty)$ with
$\lim\limits_{r\to\infty}\rho_{i}(r)=\infty \,(i=1,2)$
\begin{flushleft}
such that, for any $r>0$, there exists a finite subfamily $K_{r}\subset\mathscr{X}$ such that, for any metric space $(X,d_X)\in \mathscr{X}\backslash K_r$ and for any subset $C\subset X$ of diameter less than $r$, there exists a ¡®trivialization¡¯
\end{flushleft}
 $$t_{C}:(Y_x)_{x\in C}\to C\times Y$$ such that the restriction of $t_C$ to the fibre $Y_{x}\,(x\in C)$ is an isometry
$t_C(x):Y_{x}\to Y$, satisfying the following conditions:

\vspace{3mm}(i) for any $x,x'\in C,\,\rho_{1}(d_X(x,x'))\leq d_Y(t_{C}(x)(s(x)),t_{C}(x')(s(x')))\leq\rho_{2}(d_X(x,x'))$;

\vspace{3mm}(ii) for any two subsets $C_1,C_2\subset X\backslash K_r$ of diameter less than r with $C_{1}\cap C_{2}\neq \varnothing$, there exists an isometry $t_{C_{1}C_{2}}:Y\to Y$ such that $t_{C_1}(x)\circ t_{C_2}^{-1}(x)=t_{C_{1}C_{2}}$ for all $x\in C_{1}\cap C_{2}$.~\label{def4}
\end{definition}

For more details on fibred cofinitely-coarse embeddability, the readers can consult \cite{Orzechowski}.
\section{Proof of the main results}
In this section, we will prove Theorem \ref{th1}. Firstly, by using the notion of {\it ultrafilters} and {\it ultraproducts}, we are going to build a global isometric affine action from a family of {\it $r$-locally isometric affine actions} with $r\to+\infty$.
\begin{definition}[see \cite{Aksoy}]
Let $\mathcal{U}$ be a subset of $\mathcal{P}(\mathbb{N})$. $\mathcal{U}$ is said to be a {\it non-principal ultrafilter} on $\mathbb{N}$ if it satisfies:

(1) the empty set $\varnothing$ does not belong to $\mathcal{U}$;

(2) if $A,B\in\mathcal{P}(\mathbb{N})$ and $A\subset B$, $A\in\mathcal{U}$, then $B\in\mathcal{U}$;

(3) if $A\in\mathcal{P}(\mathbb{N})$, then $A\in\mathcal{U}$ or $\mathbb{N}\setminus A\in\mathcal{U}$;

(4) all the finite subsets of $\mathbb{N}$ do not belong to $\mathcal{U}$.~\label{def5}
\end{definition}

Let $(x_r)_{r\in\mathbb{N}}$ be a bounded real-valued  sequence. The {\it $\mathcal{U}$-limit} of $(x_r)_{r\in\mathbb{N}}$ is the unique $x\in\mathbb{R}$  denoted by $\lim_{\mathcal{U}}x_r$; it is characterized by the fact that the set $\{r\in\mathbb{N}:|x_{r}-x|\leq\varepsilon\}\in\mathcal{U}$ for any $\varepsilon >0$.

Next we shall introduce the {\it ultraproduct} of a family of Banach spaces (see \cite{Jinxiu}).

Let $(B_r)_{r\in\mathbb{N}}$ be a family of Banach spaces and $\mathcal{U}$ be a non-principal ultrafilter on $\mathbb{N}$. Consider the space $$\ell^{\infty}(\mathbb{N},(B_r)_{r\in\mathbb{N}}):=\bigg\{(a_r)\in\prod_{r\in\mathbb{N}}B_r :\sup_{r\in\mathbb{N}}\|a_r\|_{B_r}\leq\infty\bigg\}$$
endowed with the {\it seminorm} $\|(a_r)\|_{\mathcal{U}}:=\lim_{\mathcal{U}}\|a_r\|_{B_r}$ for $(a_r)\in\ell^{\infty}(\mathbb{N},(B_r)_{r\in\mathbb{N}})$,
we set
\begin{equation}\label{equ31}
(a_r)\sim_{\mathcal{U}}(b_r)\,\,\rm{if\,and \,only \,if}\,\|(a_r) -(b_r)\|_{\mathcal{U}}=0.
\end{equation}
It's easy to verify that (\ref{equ31}) defines an equivalence relation $\sim_{\mathcal{U}}$ on $\ell^{\infty}(\mathbb{N},(B_r)_{r\in\mathbb{N}})$. Then the ultraproduct $B_{\mathcal{U}}$ of the family $(B_r)_{r\in\mathbb{N}}$ with respect to a non-principal ultrafilter $\mathcal{U}$ is the closure of the space $\ell^{\infty}(\mathbb{N},(B_r)_{r\in\mathbb{N}})/\sim_{\mathcal{U}}$ with respect to the norm $\|\cdot\|_{B_{\mathcal{U}}}:=\|\cdot\|_{\mathcal{U}}$.

\begin{definition}[see \cite{Arnt}]
Let $G$ be a finitely generated group and $r$ be a non-negative real number.

\vspace{3mm}(1) Let $X$ be a set. A map $\alpha:G\times X\to X$ is said to be {\it an $r$-local action} of $G$ on $X$ if:

\vspace{3mm}i) for all $g\in G$ satisfying $l(g)<r$, $\alpha(g):X\to X$ is a bijection;

\vspace{3mm}ii) for all $g,h\in G$ such that $l(g)$, $l(h)$, $l(gh)$ are less than $r$, $$\alpha(gh)=\alpha(g)\alpha(h).$$

\vspace{3mm}(2) Let $B$ be a Banach space. A map $L:G\times B\to B$ is said to be {\it an $r$-local isometric representation} of $G$ on $B$ if $L$ is an $r$-local action of $G$ on $B$ and for all $g\in G$ such that $l(g)<r$, $L(g):B\to B$ is a linear isometry. By remark \ref{re1}, there is a map $b:G\to B$ such that, for all $g,h\in G$ satisfying $$\max\{l(g),\,l(h),\,l(gh)\}\leq r,\,\,L(g)b(h)+b(g)=b(gh),$$ and $b$ is said to be {\it an $r$-local cocycle} with respect to $L$.

\vspace{3mm}(3) Let $B$ be a Banach space. A map $\alpha:G\times B\to B$ is called {\it an $r$-local isometric affine action} of $G$ on $B$ if it satisfies $$\alpha(g)(\xi)=L(g)(\xi)+b(g)$$ for any $\xi\in B$, where $L$ is an $r$-local isometric representation and $b$ is an $r$-local cocycle for $L$.~\label{def6}

\end{definition}

The following lemma was proved by A. G. Aksoy and M. A. Khamsi in \cite{Aksoy}.
\begin{lemma}
Let $(B_r)_{r\in\mathbb{N}}$ be a family of uniformly convex Banach spaces with a common modulus of convexity. Then for any ultrafilter $\mathcal{U}$ on $\mathbb{N}$ the ultraproduct $B_{\mathcal{U}}$ is uniformly convex.~\label{lem1}
\end{lemma}

The next Lemma introduced in \cite{Arnt} is a very important tool in proving our main result. We include the proof below for the convenience of the readers.

\begin{lemma}[see \cite{Arnt}]
Let $G$ be a finitely generated group, $(B_r)_{r\in\mathbb{N}}$ be a family of Banach spaces and $B_{\mathcal{U}}$ be the ultraproduct of the family $(B_r)_{r\in\mathbb{N}}$ with respect to a non-principal ultrafilter $\mathcal{U}$ on $\mathbb{N}$. For each $r\in\mathbb{N}$, assume that $G$ admits an $r$-local isometric affine action $\alpha_{r}$ on $B_r$ with $$\alpha_{r}(g)\cdot=L_{r}(g)\cdot+b_{r}(g).$$ If, for all $g\in G,$ $(b_r (g))_{r\in\mathbb{N}}$ belongs to $B_{\mathcal{U}}$, then there exists an isometric affine action $\alpha$ of $G$ on $B_{\mathcal{U}}$ such that $$\alpha(g)\cdot=L(g)\cdot+b(g)$$ where $L$ is an isometric representation of $G$ on $B_{\mathcal{U}}$  such that $L(g)a=(L_r (g)a_r)_{r\in\mathbb{N}}$ for $g\in G$, $a=(a_r)_{r\in\mathbb{N}}\in B_{\mathcal{U}}$ and $b:G\to B_{\mathcal{U}}$ is a cocycle with respect to $L$ satisfying, for $g\in G$:$$b(g)=(b_r (g))_{r\in\mathbb{N}}.$$~\label{lem2}
\end{lemma}

\begin{proof}
First, we show that $L$ is an isometric representation of $G$ on $B_{\mathcal{U}}$.

Let $g,h\in G$, $a=(a_r)_{r\in\mathbb{N}}\in B_{\mathcal{U}}$. We have, for any $\varepsilon>0$, the set $$\{r\in\mathbb{N}:\|L_r (g)L_r (h)a_r-L_r (gh)a_r\|_{B_r} >\varepsilon\}$$ is finite and then its complement $$\{r\in\mathbb{N}:\|L_r (g)L_r (h)a_r-L_r (gh)a_r\|_{B_r} \leq\varepsilon\}$$ belongs to $\mathcal{U}$.
By the notion of $B_{\mathcal{U}}$, we obtain that $L(g)L(h)=L(gh)$ for all $g,\,h\in G$. This shows that $L$ is a representation of $G$ on $B_{\mathcal{U}}$.
Now, for $g\in G$, since for all $r$ large enough, $L_r (g)$ is an isometric isomorphism of $B_r$, it follows, by a similar argument, that $L(g)$ is an isometric isomorphism of $B_{\mathcal{U}}$. Thus, $L$ is an isometric representation of $G$ on $B_{\mathcal{U}}$.

Second, we show that $b$ is a cocycle with respect to $L$.
Let $g,h\in G$. For all $r\in\mathbb{N}$ such that $$r>\max(l(g),l(h),l(gh)),$$ we have $$b_r (gh)=L_r (g)b_r (h)+b_r (g).$$ Hence, $$b(gh)=L(g)b(h)+b(g)$$ for all $g,h\in G$, and then $b$ is a cocycle with respect to $L$.

{\noindent}It follows that the map $\alpha$ such that $$\alpha(g)\cdot=L(g)\cdot+b(g)$$ is an isometric affine action of $G$ on $B_{\mathcal{U}}$.
\end{proof}

Now we are ready to complete the proof of our main result.
\begin{proposition}
Let $G$ be a countable, residually amenable group. If a box family of $G$ admits a fibred cofinitely-coarse embedding into a
uniformly convex Banach space, then there exists a uniformly convex Banach space $B'$ and affine isometric action of $G$ on $B'$ which is proper.~\label{pro1}
\end{proposition}

\begin{proof}
Let $1\leq p<\infty$. Let $(G_n)_{n\in\mathbb{N}}$ be a sequence of normal subgroups with amenable quotients, satisfying the condition
of Definition \ref{def1}, such that the associated box family $\mathscr{X}:=\{(G/G_n, d_n):n\in\mathbb{N}\}$ admits a fibred cofinitely-coarse embedding into
a uniformly convex Banach space $B$.

We set $X_n=G/G_n$. Fix an integer $r>0$. Let $K_r$ be such as in Definition \ref{def4}. Since $K_r$ is a finite family, there exists $n_r$ such that $(X_{n_r},d_{n_r})\in\mathscr{X}\setminus K_r$. For each $C\subset X_{n_r}$ of diameter less than $r$, there exists a trivialization $t_C$ satisfying
conditions (i) and (ii) in Definition \ref{def4}.

Now, choose $n_r$ large enough such that the quotient map $\pi_{n_r}:G\to X_{n_r}$ is {\it $r$-isometric}, i.e., for each subset $Y\subset G$
of diameter less than $r$, $(\pi_{n_r})|_Y$ is an isometry onto its image.

For $z\in X_{n_r}$, we denote by $C_{z}:=\{x\in X_{n_r}:d_{X_{n_r}}(z,x)<r\}$ the $r$-ball centered in $z$ of $X_{n_r}$ and we define, for $x\in X_{n_r}$, the
following vector $c_{r}^{z}(x)$ in $B$:
\begin{align*}
c_{r}^{z}(x)=\left\{\begin{aligned}
   &t_{C_z}(z)(s(z))-t_{C_z}(zx)(s(zx)),&&{\rm{if}}\,d_{X_{n_r}}(e,x)<r \,\,({\rm{i.e.}} \,\,x\in C_e);\\
   &0,&&\rm{otherwise},\\
   \end{aligned}
   \right.
  \end{align*}
where $e$ is the identity element of $X_{n_r}$. Suppose $\rho_1,\rho_2$ are the controlling functions from Definition \ref{def4}, we obtain $$\rho_{1}(d_{X_{n_r}}(e,x))\leq \|c_{r}^{z}(x)\|_{B} \leq \rho_{2}(d_{X_{n_r}}(e,x))$$ for any $x\in C_e$.

Let us consider the map $\widetilde{b}_r:X_{n_r}\to \prod_{z\in X_{n_r}}B$, defined by, for $x\in X_{n_r}$:$$\widetilde{b}_r (x)=(c_{r}^{z}(x))_{z\in X_{n_{r}}}.$$
Now we are going to define the norm of the vectors in $\prod_{z\in X_{n_r}}B$. We refer to the amenability of $X_{n_r}$. By one of the equivalent definitions of the amenability, there exists a finitely additive and right-invariant probability measure $\mu_{n_r}:2^{X_{n_r}}\to [0,1]$. This measure determines a linear, order-preserving functional on $\ell^{\infty}(X_{n_r})$ which is invariant under composition with right multiplication in the group. For $\xi=(\xi_z)_{z\in X_{n_{r}}} \in\prod_{z\in X_{n_r}}B,\,\xi_z \in B$ define $$\| \xi\|_p=\bigg(\int_{X_{n_r}}\|\xi_z \|_B^p d\mu_{n_r}(z)\bigg)^{\frac{1}{p}}.$$
Hence, for $x\notin C_e,\,\widetilde{b}_r$ vanishes, and for $x\in C_e$,
$$\rho_{1}(d_{X_{n_r}}(e,x))\leq \|\widetilde{b}_{r}(x)\|_{p}=\bigg(\int_{X_{n_r}}\|c_{r}^{z}(x)\|_B^p d\mu_{n_r}(z)\bigg)^{\frac{1}{p}} \leq \rho_{2}(d_{X_{n_r}}(e,x)).$$

For $x\in C_e$, we define $\widetilde{\sigma}_r (x):\prod_{z\in X_{n_r}}B \to \prod_{z\in X_{n_r}}B$ by, for all $\xi =(\xi_z)_{z\in X_{n_{r}}}$:
$$\widetilde{\sigma}_r(x)(\xi)= \begin{cases} (t_{C_{z}C_{zx}}(\xi_{zx}))_{z\in X_{n_{r}}}, & {\rm{if}} \,\,x\in C_e;\\
   \xi, &\rm{otherwise}. \end{cases}$$
We claim that $\widetilde{\sigma}_r$ is an $r$-local isometric representation and $\widetilde{b}_r$ is an $r$-local cocycle with respect to $\widetilde{\sigma}_r$. First, it is clear that $\widetilde{\sigma}_r(x)$ is an isometric isomorphism for all $x\in X_{n_r}$; for $z\in X_{n_{r}}$, it follows from Definition \ref{def4} (ii) that $t_{C_{z}C_{zy}}\circ t_{C_{zy}C_{zyx}}=t_{C_{z}C_{zyx}}$ for all $x,y\in C_e$, with $d_{X_{n_r}}(e,yx)<r$.
Suppose $x,y,yx \in C_e$, for any $\xi \in \prod_{z\in X_{n_r}}B$
$$\widetilde{\sigma}_r(yx)(\xi)=(t_{C_{z}C_{zyx}}(\xi_{zyx}))_{z\in X_{n_{r}}}=(t_{C_{z}C_{zy}} \circ t_{C_{zy}C_{zyx}}(\xi_{zyx}))_{z\in X_{n_{r}}}$$
and
\begin{align*}
\widetilde{\sigma}_r(y)\widetilde{\sigma}_r(x)(\xi)&=\widetilde{\sigma}_r(y)(t_{C_{z}C_{zx}}(\xi_{zx}))_{z\in X_{n_{r}}}\\
&=(t_{C_{z}C_{zy}} \circ t_{C_{zy}C_{zyx}}(\xi_{zyx}))_{z\in X_{n_{r}}}.
\end{align*}
Hence, $\widetilde{\sigma}_r(yx)=\widetilde{\sigma}_r(y)\widetilde{\sigma}_r(x)$. Thus, $\widetilde{\sigma}_r$ is an $r$-local isometric representation of $X_{n_r}$ on $\prod_{z\in X_{n_r}}B$.

\vspace{2mm}{\noindent}Second, for $x,y\in C_e$ with $d_{X_{n_r}}(e,yx)<r$, we have $$\widetilde{\sigma}_r(y)(\widetilde{b}_r(x))+\widetilde{b}_r(y)=\widetilde{b}_r(yx).$$
In fact,
\begin{align*}
t_{C_{z}C_{zy}}(c_{r}^{zy}(x))&=t_{C_{z}C_{zy}}(t_{C_{zy}}(zy)(s(zy))-t_{C_{zy}}(zyx)(s(zyx)))\\
&=t_{C_{z}C_{zy}}\circ t_{C_{zy}}(zy)(s(zy))-t_{C_{z}C_{zy}}\circ t_{C_{zy}}(zyx)(s(zyx))\\
&=t_{C_{z}}(zy)(s(zy))-t_{C_{z}}(zyx)(s(zyx)).
\end{align*}
The last equality is by Definition \ref{def4} (ii).
Thus,$$t_{C_{z}C_{zy}}(c_{r}^{zy}(x))+c_{r}^{z}(y)=t_{C_{z}}(z)(s(z))-t_{C_{z}}(zyx)(s(zyx))=c_{r}^{z}(yx).$$
It follows that:
\begin{align*}
&\widetilde{\sigma}_r(y)(\widetilde{b}_r(x))+\widetilde{b}_r(y)\\
=&(t_{C_{z}C_{zy}}(c_{r}^{zy}(x))+c_{r}^{z}(y))_{z\in X_{n_{r}}}\\
=&(c_{r}^{z}(yx))_{z\in X_{n_{r}}}\\
=&\,\,\widetilde{b}_r(yx),
\end{align*}
which proves that $\widetilde{b}_r$ is an $r$-local cocycle with respect to $\widetilde{\sigma}_r$.

Now, let $\sigma_{r}:=\widetilde{\sigma}_r \circ \pi_{n_r}$, $b_{r}:=\widetilde{b}_r \circ \pi_{n_r}$ be the lifts of $\widetilde{\sigma}_r$ and $\widetilde{b}_r$ to the $r$-ball $\{g\in G:d_{G}(e_{G},g)<r\}$
of $G$ and define $\sigma_{r}=\rm{Id}$, $b_{r}=0$ outside the ball. Then, obviously, $\sigma_{r}$ is an $r$-local isometric representation of $G$ on $\prod_{z\in X_{n_r}}B$, $b_r$ is an $r$-local cocycle with respect to $\sigma_{r}$. Then the map $\alpha_r$ satisfying $$\alpha_{r}(g)(\xi):=\sigma_{r}(g)(\xi)+b_{r}(g)$$ for $\xi \in\prod_{z\in X_{n_r}}B$ is an $r$-local isometric affine action of $G$ on $\prod_{z\in X_{n_r}}B$ and we have, for $g\in G$ with $d_{G}(e_G,g)<r$:
$$\rho_{1}(d_{G}(e_{G},g))\leq \|b_{r}(g)\|_{p}\leq \rho_2(d_{G}(e_{G},g)).$$
By Lemma \ref{lem2}, from these local isometric affine actions, we can build a global isometric affine action of $G$.

Let $\mathcal{U}$ be a non-principal ultrafilter on $\mathbb{N}$, and let $B_{\mathcal{U}}$ be the ultraproduct of the family
\vspace{1.5mm}$(\prod_{X_{n_r}}B)_{r\in \mathbb{N}}$
with respect to $\mathcal{U}$.
For each $r\in \mathbb{N}$, $\alpha_r$ is an $r$-local isometric affine action of
\vspace{1.5mm}$G$ on $\prod_{X_{n_r}}B$ and for $g\in G$, $\|b_{r}(g)\|_{p}\leq \rho_2(d_{G}(e_{G},g))$ for all $r\in \mathbb{N}$, so $(b_{r}(g))_{r\in \mathbb{N}}\in B_{\mathcal{U}}$. Hence, by Lemma \ref{lem2}, there exists an isometric affine action $\alpha$ of $G$ on $B_{\mathcal{U}}$ such that $b:g\mapsto (b_{r}(g))_{r\in\mathbb{N}}$ is a cocycle for $\alpha$. In addition, since for any $g\in G$, $\rho_{1}(d_{G}(e_{G},g))\leq \|b_{r}(g)\|_{p}$ for all $r$ large enough, we have $\rho_{1}(d_{G}(e_{G},g))\leq \|b(g)\|_{B_{\mathcal{U}}}$ for all $g\in G$. Then, by the condition \eqref{equ21}, $\alpha$ is proper.

Now we will show that $B_{\mathcal{U}}$ is a uniformly convex Banach space.
By Day's Theorem (i.e.$\,$Theorem \ref{th2}), we obtain that $\prod_{X_{n_r}}B$ is uniformly convex. Let $E_{i}:=\prod_{X_{n_{r_i}}}B$, $i\in \mathbb{N}$. Then $\{E_{i}\}_{i\in \mathbb{N}}$ is a family of uniformly convex Banach spaces with a common modulus of convexity. Indeed, without loss of generality we may assume that  $G/G_1$ is countably infinite. It follows that for every  $n\in \mathbb{N}$, $G/G_n$ is countably infinite. For each $i\in\mathbb{N}$, $E_i$ is isomorphic to $\bigoplus_{j\in\mathbb{N}}B$. Hence, $\{E_{i}\}_{i\in \mathbb{N}}$ have a common modulus of convexity. By Lemma \ref{lem1}, we know that $B_{\mathcal{U}}$ is a uniformly convex Banach space. It follows that there exists a proper action by affine isometries of $G$ on a uniformly convex Banach space. Let $B'=B_{\mathcal{U}}$, then the proof is complete.

\end{proof}

Now, we are moving on to the converse statement. The proof essentially applies the method used to prove Theorem 4.3 in \cite{Orzechowski} which still goes through if we replace the Hilbert space with a uniformly convex Banach space. We include the proof below.
\begin{proposition}
Let $G$ be a countable, residually amenable group. If $G$ admits a proper isometric affine action on some uniformly convex Banach space, then any box family of $G$ admits a fibred cofinitely-coarse embedding into some uniformly convex Banach space.~\label{pro2}
\end{proposition}

\begin{proof}
Let $(G_{n})_{n\in \mathbb{N}}$ be a nested sequence of normal subgroups of $G$ with trivial intersection and $\mathscr{X}:=\{(G/G_n, d_n):n\in\mathbb{N}\}$ the corresponding box family. By assumption, there exists a uniformly convex Banach space $B$ and a proper action of $G$ on $B$ by affine isometries, i.e. there exists a homomorphism $\alpha:G\to \rm{AffIsom}(B)$. For $g\in G$, $\xi\in B$, $\alpha(g)$ can be written as: $$\alpha(g)(\xi)=L(g)(\xi)+b(g)$$ where $L$ is an isometric representation of $G$ and $b$ is a cocycle with respect to $L$ satisfying $$b(gh)=L(g)(b(h))+b(g)$$ for all $g,h\in G$. Note that when $g,h=e_G$, we have $b(e_G)=0$.

For each $n\in \mathbb{N}$, there exists a natural action of $G_n$ on $G\times B$, explicitly we have $$g(x,\xi):=(gx,\alpha(g)(\xi))$$
for $g\in G_n$, $x\in G$, $\xi\in B$. The orbit of $(x,\xi)$ will be denoted by $[(x,\xi)]$. Let $\pi_{n}:G\to G/G_{n}$ be a quotient map. If $\pi_{n}(a)=G_{n}a=[a]$ is an element of $G/G_n$, the action can be restricted to $G_{n}a\times B$ and we define $$B_{[a]}:=(G_{n}a\times B)/G_n$$ which means that $B_{[a]}$ is the orbit space. We can define a metric on $B_{[a]}$:$$d_{B_{[a]}}([(x,\xi)],[(x',\xi')]):=\|\xi'-\alpha(x'x^{-1})(\xi)\|_B.$$
Next, we will show that $d_{B_{[a]}}$ is well-defined. Let $x,x',x''\in G_{n}a$, $\xi,\xi',\xi''\in B$, $g,g'\in G$. Then
\begin{align*}
&d_{B_{[a]}}([(gx,\alpha(g)(\xi))],[(g'x',\alpha(g')(\xi'))])\\
=&\|\alpha(g')(\xi')-\alpha(g'x'x^{-1}g^{-1})\alpha(g)(\xi)\|_B=\|\alpha(g')(\xi')-\alpha(g'x'x^{-1})(\xi)\|_B\\
=&\|\xi'-\alpha(x'x^{-1})\|_B=d_{B_{[a]}}([(x,\xi)],[(x',\xi')]),
\end{align*}
because $\alpha(g')$ is an affine isometry. It follows that the definition of $d_{B_{[a]}}$ is independent of the choice of representative. We shall verify the basic conditions that the metric should satisfy:\\
i) $d_{B_{[a]}}\geq0$. Moreover, if $d_{B_{[a]}}([(x,\xi)],[(x',\xi')])=0$, then $\xi'=\alpha(x'x^{-1})(\xi)$. So we have $[(x',\xi')]=[(x',\alpha(x'x^{-1})(\xi))]=[(x'x^{-1}x,\alpha(x'x^{-1})(\xi))]=[(x,\xi)]$. Conversely, if $[(x,\xi)]=[(x',\xi')]$, then there exists $g\in G$, such that $g(x,\xi)=(gx,\alpha(g)(\xi))=(x',\xi')$. Thus, $\xi'=\alpha(x'x^{-1})(\xi)$, $d_{B_{[a]}}([(x,\xi)],[(x',\xi')])=\|\xi'-\alpha(x'x^{-1})(\xi)\|_B=0$.

\noindent ii) Obviously, the symmetry holds.

\noindent iii) The triangle inequality:
\begin{align*}
&d_{B_{[a]}}([(x,\xi)],[(x',\xi')])+d_{B_{[a]}}([(x',\xi')],[(x'',\xi'')])\\
=&\|\xi'-\alpha(x'x^{-1})(\xi)\|_B+\|\xi''-\alpha(x''x'^{-1})(\xi')\|_B\\
=&\|\xi''-\alpha(x''x'^{-1})(\xi')\|_B+\|\alpha(x''x'^{-1})\xi'-\alpha(x''x'^{-1})\alpha(x'x^{-1})(\xi)\|_B\\
\geq&\|\xi''-\alpha(x''x^{-1})(\xi)\|_B\\
=&d_{B_{[a]}}([(x,\xi)],[(x'',\xi'')]).
\end{align*}
In the following statement, we are going to show that $B_{[a]}$ defined above corresponds to $Y_x$ in the notion of fibred cofinitely-coarse embeddings.
For any fixed $a_{0}\in G_{n}a$, we can define a map from $B_{[a]}$ onto $B$, as follows:
\begin{equation}\label{equ1}
[(x,\xi)]\mapsto \alpha(a_{0}x^{-1})(\xi).
\end{equation}
This map is isometric. Indeed, $$\|\alpha(a_{0}x'^{-1})(\xi')-\alpha(a_{0}x^{-1})(\xi)\|_B=\|\xi'-\alpha(x'x^{-1})(\xi)\|_B=d_{B_{[a]}}([(x,\xi)],[(x',\xi')]),$$ and each $\xi\in B$ is the image of $[(a_{0},\xi)]$.

\vspace{1mm}So, we have obtained the field $(B_{[a]})_{[a]\in G/G_n}$ of metric spaces isometric to $B$. Letting\vspace{1.5mm}
$n$ vary, we get $(B_x)_{x\in \bigcup\mathscr{X}}$. For $[a]\in G/G_n$, set $s([a]):=[(a,b(a))]\in B_{[a]}$. Since $$[(ga,b(ga))]=[(ga,L(g)(b(a))+b(g))]=[(ga,\alpha(g)(b(a)))]=[(a,b(a))],$$ this is a well-defined section on $\bigcup\mathscr{X}$.

For any $r>0$. Let $n_{r}\in\mathbb{N}$ be the smallest number such that, for $n\geq n_{r}$, $B_{G}(e_{G},3r)\cap G_n ={e_G}$. Let now $K_r :=\{G/G_n :n<n_r\}$. Choose a space $(G/G_n ,d_n)\in\mathscr{X}\setminus K_r$ and a subset $C\subset G/G_n$ with diameter less than $r$. We will define a trivialization: $$t_C:\bigsqcup_{x\in C}B_x \to C\times B.$$
Let $g\in G$ be an arbitrary point (``basepoint'') of $\pi_{n}^{-1}(C)$. Then each $[a]\in C$ has a unique lift $a_{0}\in G$ such that $[a_0]=[a]$ and $d(g,a_0)<r$. Moreover, the lifting map $[a]\mapsto a_0$ is isometric on $C$. We define $t_C ([a]):B_{[a]}\to B$ to be the map given in \eqref{equ1} with respect to our choice of $a_0$. Now, for any $[a],[a']\in C$, we have
\begin{align*}
&\|t_C ([a])(s([a]))-t_C ([a'])(s([a']))\|_B\\
=&\|t_C ([a])([(a,b(a))])-t_C ([a'])([(a',b(a'))])\|_B\\
=&\|\alpha(a_0 a^{-1})(b(a))-\alpha(a'_0 a'^{-1})(b(a'))\|_B\\
=&\|L(a_0 a^{-1})(b(a))+b(a_0 a^{-1})-(L(a'_0 a'^{-1})(b(a'))+b(a'_0 a'^{-1}))\|_B\\
=&\|b(a_0 a^{-1}a_0)-b(a'_0 a'^{-1}a'_0)\|_B=\|b(a_0)-b(a'_0)\|_B\\
=&\|L(a_{0}^{-1})(b(a_0))+b(a_{0}^{-1})-(L(a_{0}^{-1})(b(a'_0))+b(a_{0}^{-1}))\|_B\\
=&\|b(a_{0}^{-1}a_0)-b(a_{0}^{-1}a'_{0})\|_B\\
=&\|b(e_G)-b(a_{0}^{-1}a'_{0})\|_B\\
=&\|b(a_{0}^{-1}a'_{0})\|_B.
\end{align*}
If we define $\rho_1 (x):=\min\{\|b(g)\|_B:g\in G,\,l(g)\geq x\}$ and $\rho_2 (x):=\max\{\|b(g)\|_B:g\in G,\,l(g)\leq x\}$ for all $x\in [0,\infty)$, then
\begin{align*}
\|t_C ([a])(s([a]))-t_C ([a'])(s([a']))\|_B=\|b(a_{0}^{-1}a'_{0})\|_B\in&[\rho_1 (d(a_0,a'_0)),\rho_2 (d(a_0,a'_0))]\\
=&[\rho_1 (d_n ([a],[a'])),\rho_2 (d_n ([a],[a']))].
\end{align*}
Note that the controlling functions do not depend on $n\in \mathbb{N}$. The condition $\lim\limits_{r\to \infty}\rho_1 (r)=\infty$ is satisfied due to the fact that the action of $G$ on $B$ is proper. In the trivial situation when $G$ is finite, $\rho_1$ needs a slight modification to avoid taking the value $\infty$, because the set $\{\|b(g)\|_B:g\in G,\,l(g)\geq x\}$ is empty for large $x$. However, this is not a significant problem.

It only suffices to show that, whenever $C_1,C_2 \subset G/G_n$ have diameters less than $r$ and $[a]\in C_1 \cap C_2$, the corresponding trivializations $t_{C_1} ([a])$, $t_{C_2} ([a])$ differ only by an isometric map $t_{C_1} ([a])\circ t_{C_2}^{-1} ([a]):B\to B$. Indeed, there exist $g_0 \in G$, $g_1 ,g_2 ,\in\pi_{n}^{-1}(C_1 \cap C_2)$ such that $\pi_{n}(g_1)=\pi_{n}(g_2)\in C_1 \cap C_2$ and $g_1 =g_0 g_2$. Let $g_1,\,g_2$ be ``basepoints" and $a_1,a_2$ corresponding lifts of $[a]$. Then, for any $\xi\in B$, we have
$$(t_{C_1} ([a])\circ t_{C_2}^{-1} ([a]))(\xi)=t_{C_1} ([a])([(a_2,\xi)])=\alpha(a_1 a_{2}^{-1})(\xi)=\alpha(g_0)(\xi),$$
because each $\xi\in B$ is the image of $[(a_2,\xi)]$. By assumption, the map $\alpha(g_0):B\to B$ is an isometry, which completes the proof.

\end{proof}

By Proposition \ref{pro1}$\,\,$and \ref{pro2}, we obtain a sufficient and necessary condition for a residually amenable group to admits a proper isometric affine action on some uniformly convex Banach space. That is, Theorem \ref{th1} is proved.


\vskip 1cm

\noindent \noindent Guoqiang Li\\
College of Mathematics and Statistics,\\
Chongqing University (at Huxi Campus),\\
Chongqing 401331, P. R. China\\
E-mail: \url{guoqiangli@cqu.edu.cn}\\

\noindent \noindent Xianjin Wang\\
College of Mathematics and Statistics,\\
Chongqing University (at Huxi Campus),\\
Chongqing 401331, P. R. China\\
E-mail: \url{xianjinwang@cqu.edu.cn}\\
\end{document}